\documentclass[twoside]{amsart}

\usepackage[T1]{fontenc}
\usepackage{amssymb,mathtools,bm,mathrsfs,stmaryrd}
\usepackage{tikz-cd}
\usepackage{hyperref}
\usepackage{enumerate}
\usepackage{turnstile}
\usepackage{marvosym}
\usepackage{colonequals}

\newtheorem*{rep@theorem}{\rep@title}
\newcommand{\newreptheorem}[2]{%
\newenvironment{rep#1}[1]{%
 \def\rep@title{#2 \ref{##1}}%
 \begin{rep@theorem}}%
 {\end{rep@theorem}}}

\numberwithin{equation}{section}

\newtheorem{theorem}[equation]{Theorem}
\newreptheorem{thm}{Theorem}
\newtheorem{proposition}[equation]{Proposition}
\newtheorem{lemma}[equation]{Lemma}
\newtheorem{corollary}[equation]{Corollary}

\theoremstyle{definition}
\newtheorem{definition}[equation]{Definition}

\newtheorem{example}[equation]{Example}
\newtheorem{discussion}[equation]{}

\theoremstyle{remark}

\makeatletter
\renewcommand\p@enumii{}
\makeatother

\tikzcdset{
  column sep/wide/.initial=+6em,
  column sep/wider/.initial=+7.2em,
}

\newcommand{\nospacepunct}[1]{\makebox[0pt][l]{\,#1}}

\renewcommand{\epsilon}{\varepsilon}
\renewcommand{\phi}{\varphi}
\renewcommand{\theta}{\vartheta}

\newcommand{\Inj}[1]{#1\text{-}\operatorname{Inj}}

\newcommand{\kbfcat}[1]{{\rm K}_{b}^{f}(#1)}

\newcommand{\dcat}[2][]{{\rm D}_{#1}(#2)}
\newcommand{\dfbcat}[1]{{\rm D}_{b}^{f}(#1)}
\newcommand{\dfmcat}[1]{{\rm D}_{-}^{f}(#1)}
\newcommand{\dfpcat}[1]{{\rm D}_{+}^{f}(#1)}
\newcommand{\Perf}[1]{\operatorname{Perf}(#1)}

\DeclareMathOperator{\add}{add}
\DeclareMathOperator{\Add}{Add}
\newcommand{\cat}[1]{{\mathcal{#1}}}
\DeclareMathOperator{\ctr}{Z}
\DeclareMathOperator{\codim}{codim}
\DeclareMathOperator{\cogin}{cogin}
\DeclareMathOperator{\coh}{H}
\newcommand{\complete}[1]{\widehat{#1}}
\newcommand{\env}[1]{{#1}^e}
\newcommand{\gentime}{\text{\Clocklogo}}
\DeclareMathOperator{\gin}{gin}
\DeclareMathOperator{\gldim}{gldim}
\DeclareMathOperator{\height}{height}
\DeclareMathOperator{\Hom}{Hom}
\DeclareMathOperator{\id}{id}
\newcommand{\Kos}{K}
\DeclareMathOperator{\level}{level}
\newcommand{\Lotimes}{\otimes^{\mathbf{L}}}
\DeclareMathOperator{\Max}{Max}
\newcommand{\op}[1]{#1^{op}}
\DeclareMathOperator{\Prod}{Prod}
\DeclareMathOperator{\RHom}{\mathbf{R}Hom}
\DeclareMathOperator{\smd}{smd}
\DeclareMathOperator{\Spec}{Spec}
\DeclareMathOperator{\Supp}{Supp}
\DeclareMathOperator{\supp}{supp}
\newcommand{\susp}{\Sigma}
\DeclareMathOperator{\thick}{thick}

\newcommand{\BZ}{\mathbb{Z}}

\newcommand{\set}[2]{\left\{#1 \,\middle|\, #2\right\}}

\title[Local to global principles for level]{Local to global principles\\for generation time over\\commutative noetherian rings}

\author[J.~C.~Letz]{Janina C.~Letz}
\email{jletz@math.uni-bielefeld.de}
\address{Fakult\"at f\"ur Mathematik,
         Universit\"at Bielefeld,
         33501 Bielefeld,
         Germany}



\begin{document}

\begin{abstract}
In the derived category of modules over a commutative noetherian ring a complex $G$ is said to generate a complex $X$ if the latter can be obtained from the former by taking summands and finitely many cones. The number of cones required in this process is the generation time of $X$. In this paper we present some local to global type results for computing this invariant, and discuss applications.
\end{abstract}

\maketitle

\section{Introduction}

The goal of this paper is to develop techniques for computing generation time of complexes over commutative noetherian rings. In the derived category of a commutative noetherian ring $R$, a complex $G$ \emph{generates} a complex $X$ if the latter is obtained from the former by taking suspensions, finite direct sums, direct summands and (mapping) cones. The minimal number of cones required is the \emph{generation time}, or \emph{level}, of $X$ with respect to $G$ and denoted $\level_R^G(X)$. Bondal and van den Bergh introduced the notion of generation in \cite{Bondal/VanDenBergh:2003}; see also \cite{Rouquier:2008}. The notation and terminology of level are adopted from \cite{Avramov/Buchweitz/Iyengar/Miller:2010}. 

Level has connections to other, more familiar, invariants. When $G = R$ and $X$ is a finitely generated module, level of $X$ with respect to $R$ is the projective dimension of $X$; see \cite{Christensen:1998}. When $X$ is a complex with finitely generated total homology, $\level_R^R(X)$ is bounded above by the projective dimension of $X$---that is the minimal length of a projective resolution---but typically level is smaller. 

When $R$ is a semilocal ring with Jacobson radical $J(R)$, level of a module with respect to $G = R/J(R)$ is the Loewy length. Level with respect to $G$ gives an extension of this notion to complexes. 

Unlike projective dimension and Lowey length, level behaves better under functors of derived categories. This is because such a functor need not map projectives to projectives or semisimple modules to semisimple modules. This flexibility afforded by level becomes useful. 

Despite their utility, there are few results on the behavior of level even under functors induced by a change of rings. This paper tracks the behavior of level under standard commutative algebra operations, notably localizations and completions. 

The main result considers the localizations of a commutative noetherian ring:

\begin{repthm}{localGlobalFiniteLevel}
Let $R$ be a commutative noetherian ring, and $G$ and $X$ complexes of $R$-modules with bounded and degreewise finitely generated homology. If $\level_{R_\mathfrak{p}}^{G_\mathfrak{p}}(X_\mathfrak{p}) < \infty$ for all prime ideals $\mathfrak{p}$ of $R$, then $\level_R^G(X) < \infty$.
\end{repthm}

This statement should be compared with, and extends the result of Bass and Murthy \cite[Lemma~4.5]{Bass/Murthy:1967} that a finitely generated module has finite projective dimension if it has finite projective dimension locally. 

The theorem has the following applications:

From Hopkins' \cite[Theorem~11]{Hopkins:1987} and Neeman's \cite[Lemma~1.2]{Neeman:1992b} result about perfect complexes, we deduce in Theorem~\ref{finInjDimCxSupportGeneration} that for complexes $X$, $Y$ of finite injective dimension with degreewise finitely generated and bounded homology, $X$ generates $Y$ if and only if the support of $X$ contains the support of $Y$. 

One can also track the behavior of proxy smallness, introduced in \cite{Dwyer/Greenlees/Iyengar:2006a,Dwyer/Greenlees/Iyengar:2006b}. A complex $X$ is \emph{proxy small} if it generates a perfect complex $Y$ with the same support as $X$. We prove that $X$ is proxy small if and only if it is proxy small locally; see Proposition~\ref{localGlobalProxySmall}. By \cite{Pollitz:2019}, proxy small objects in $\dfbcat{R}$ characterize whether a local ring $R$ is complete intersection. We conclude the proxy smallness property can be used to characterize \emph{locally} complete intersection rings. 

The main tool to prove Theorem~\ref{localGlobalFiniteLevel} is a converse coghost lemma proved by Oppermann and {\v{S}}\v{t}ov{\'{\i}}{\v{c}}ek \cite{Oppermann/Stovicek:2012}. A map is \emph{$G$-coghost} if it cannot be detected by post-composition with any suspension of $G$. The \emph{coghost index} of $X$ with respect to $G$ is the minimal number $n$ for which every $n$-fold composition of $G$-coghost maps that ends at $X$ is zero. It is well-known that coghost index is less than or equal to level; see \cite{Kelly:1965}. By the converse coghost lemma one has an equality in the bounded derived category of a commutative noetherian ring.

As a consequence of the converse coghost lemma, we show that computing level reduces to complete local rings. 

Section~\ref{sec:levelCoghost} recalls the definitions of level and coghost index. There we state the converse coghost lemma and discuss some aspects of the proof, as well as deduce a converse \emph{ghost} lemma whenever the ring has a dualizing complex. Localizations are discussed in Section~\ref{sec:localGlobal}, applications are discussed in Section~\ref{sec:application} and proxy smallness in Section~\ref{sec:VirtualProxySmallness}. 

\section{Level and coghost}\label{sec:levelCoghost}

In this section, we recall the definitions for level and (co)ghost index from \cite{Bondal/VanDenBergh:2003} and \cite{Avramov/Buchweitz/Iyengar/Miller:2010}. Then the converse coghost lemma is stated, and a converse ghost lemma is proved. 

\subsection*{Level} Let $\cat{T}$ be a triangulated category and $\cat{C}$ a subcategory of $\cat{T}$. Then 
\begin{enumerate}
\item $\add(\cat{C})$ denotes the smallest strictly full subcategory of $\cat{T}$ containing $\cat{C}$ that is closed under finite direct sums and suspensions, and
\item $\smd(\cat{C})$ denotes the smallest strictly full subcategory of $\cat{T}$ containing $\cat{C}$ that is closed under direct summands. 
\end{enumerate}
If $\cat{C}_1$ and $\cat{C}_2$ are subcategories of $\cat{T}$, then $\cat{C}_1 \star \cat{C}_2$ is the strictly full subcategory containing all objects $X$, such that there exists an exact triangle
\[
Y \to X \to Z \to \susp Y
\]
in $\cat{T}$ with $Y \in \cat{C}_1$ and $Z \in \cat{C}_2$. Set
\[
\cat{C}_1 \diamond \cat{C}_2 \colonequals \smd(\add(\cat{C}_1) \star \add(\cat{C}_2))\,.
\]
Recall $\cat{S}$ is a \emph{thick} subcategory of $\cat{T}$, if it is a triangulated subcategory and is closed under direct summands. 

\begin{definition}
For a subcategory $\cat{C}$ of $\cat{T}$ the \emph{$n$-th thickening} is
\[
\thick_\cat{T}^n(\cat{C}) \colonequals \begin{cases}
\{0\} & n=0 \\
\smd(\add(\cat{C})) & n=1 \\
\thick_\cat{T}^{n-1}(\cat{C}) \diamond \thick_\cat{T}^1(\cat{C}) & n \geq 2\,.
\end{cases}
\]
The \emph{level} of an object $X$ in $\cat{T}$ with respect to $\cat{C}$ is
\[
\level_\cat{T}^\cat{C}(X) \colonequals \inf\set{n \geq 0}{X \in \thick_\cat{T}^n(\cat{C})}\,.
\]
\end{definition}

The union of all thickenings of a subcategory $\cat{C}$ is the smallest thick subcategory of $\cat{T}$ containing $\cat{C}$, denoted $\thick_\cat{T}(\cat{C})$. Thus the thickenings give a filtration of $\thick_\cat{T}(\cat{C})$. Level behaves nicely with respect to direct sums and exact triangles; see \cite[Lemma~2.4]{Avramov/Buchweitz/Iyengar/Miller:2010}. In the following, we are interested in the generation by a single object $G$, and in this case we write $\level^G$ for $\level^{\{G\}}$. 

\begin{discussion} \label{levelThickSubcat}
Let $\cat{T}$ be a triangulated category and $\cat{S}$ a thick subcategory. For $\cat{C}$ a subcategory of $\cat{S}$, and $X$ in $\cat{S}$ one has
\[
\thick_\cat{T}^n(\cat{C}) = \thick_\cat{S}^n(\cat{C}) \quad \text{and} \quad \level_\cat{T}^\cat{C}(X) = \level_\cat{S}^\cat{C}(X)\,.
\]
\end{discussion}

Given an exact functor $\mathsf{f} \colon \cat{S} \to \cat{T}$ of triangulated categories, for a subcategory $\cat{C}$ of $\cat{C}$ and an object $X$ of $\cat{S}$ we have the inequality
\begin{equation}\label{eq:levelExactFunctor}
\level_\cat{S}^\cat{C}(X) \geq \level_\cat{T}^{\mathsf{f}(\cat{C})}(\mathsf{f}(X))\,.
\end{equation}

\subsection*{Coghost} To show that \eqref{eq:levelExactFunctor} is, for certain functors $\mathsf{f}$, an equality, we utilize the connection between level and coghost index given by the converse coghost lemma.

\begin{definition}
Let $\cat{C}$ be a full subcategory of $\cat{T}$. A morphism $f\colon X \to Y$ is \emph{$\cat{C}$-coghost}, if the induced map
\[
f^* \colon \Hom_\cat{T}(Y,\susp^n C) \to \Hom_\cat{T}(X,\susp^n C)
\]
is zero for any $C \in \cat{C}$ and any integer $n$.

A map $f\colon X \to Y$ is \emph{$n$-fold $\cat{C}$-coghost}, if it can be written as a composition of $n$ $\cat{C}$-coghost maps. The \emph{coghost index with respect to $\cat{C}$} is defined by
\[
\cogin_\cat{T}^\cat{C}(Y) \colonequals \inf\set{n \geq 0}{\text{all $n$-fold $\cat{C}$-coghost maps } X \to Y \text{ are zero}}\,.
\]
\end{definition}

Just as for level we are interested in the case where $\cat{C}$ consists of one object $G$. Then we write $\cogin^G$ for $\cogin^{\{G\}}$. 

Unlike level, coghost index may depend on the ambient category. More precisely given a thick subcategory $\cat{S} \subset \cat{T}$, one has
\[
\cogin^\cat{C}_\cat{S}(X) \leq \cogin^\cat{C}_\cat{T}(X)
\]
for $\cat{C}$ a subcategory of $\cat{S}$, and $X$ in $\cat{S}$. 

Similar to coghost maps, \emph{ghost} maps are the maps that become zero by pre-composition with any map from a suspension of an object in $\cat{C}$. Then the \emph{ghost index} is
\[
\gin_\cat{T}^\cat{C}(X) \colonequals \inf\set{n \geq 0}{\text{all $n$-fold $\cat{C}$-ghost maps } X \to Y \text{ are zero}}\,.
\]
The ghost and coghost maps are dual to each other in the sense that $f$ is ghost in $\cat{T}$ if and only if $f$ is coghost in $\op{\cat{T}}$. The same holds for ghost and coghost index:
\begin{equation}\label{eq:dualCoGhost}
\gin_\cat{T}^\cat{C}(X) = \cogin_{\op{\cat{T}}}^\cat{C}(X)\,.
\end{equation}

It is well known that level and coghost index always satisfy
\begin{equation}\label{eq:coghostLemma}
\cogin_\cat{T}^\cat{C}(X) \leq \level_\cat{T}^\cat{C}(X)\,.
\end{equation}
This is called the coghost lemma; see \cite{Kelly:1965}. The same inequality holds when replacing $\cogin$ by $\gin$. 

\begin{discussion} \label{generalConverseCoGhostLemma}
We do not know whether level and coghost index (or ghost index) are equal in an arbitrary triangulated category. Some partial converses are known:

If every object has a left (right) approximation by a direct sum of suspension of objects in $\cat{C}$ then the converse coghost (ghost) lemma holds; see \cite[Lemma~1.3]{Beligiannis:2008}. 

In the bounded derived category of a ring $R$, Christensen \cite{Christensen:1998} showed that level and ghost index with respect to $R$ are the same. 

The case that is relevant for this paper is due to Oppermann and {\v{S}}\v{t}ov{\'{\i}}{\v{c}}ek; see \cite{Oppermann/Stovicek:2012}. They show that in the bounded derived categories of Noether algebras level and coghost index agree for any generator. This result is discussed below. We focus on commutative noetherian rings; for a discussion on Noether algebras, see \ref{NoetherAlgebra}. 
\end{discussion}

\subsection*{Converse coghost lemma} For a noetherian ring $R$, the derived category of left $R$-modules is denoted $\dcat{R}$. The full subcategory of complexes with finitely generated total homology, that is $\coh(X) = \bigoplus_{i \in \BZ} \coh_i(X)$ is finitely generated, is denoted $\dfbcat{R}$. This is a thick subcategory of $\dcat{R}$. By \ref{levelThickSubcat}, it does not matter whether $\thick^n(G)$ and $\level^G(X)$ are calculated in $\dcat{R}$ or $\dfbcat{R}$ (or any bounded above/below derived category of (finitely generated) left $R$-modules) for objects $G$ and $X$ in $\dfbcat{R}$. They only depend on $R$, so we write
\[
\thick_R^n(G) \colonequals \thick_{\dcat{R}}^n(G) \quad\text{and} \quad \level_R^G(X) \colonequals \level_{\dcat{R}}^G(X)\,.
\]

In the following $R$ is a commutative noetherian ring, unless mentioned otherwise. 

\begin{discussion} \label{converseCoghostLemma} 
For the derived category of complexes with finitely generated total homology, \cite[Theorem~24]{Oppermann/Stovicek:2012} established a converse coghost lemma: Let $R$ be a commutative noetherian ring and $G$ an object in $\dfbcat{R}$. Then for any $X$ in $\dfbcat{R}$
\[
\cogin_{\dfbcat{R}}^G(X) = \level_R^G(X)\,.
\]
\end{discussion}

Given $n = \level_R^G(X)$ this result guarantees the existence of a nonzero $(n-1)$-fold $G$-coghost map with target $X$. This is useful, since under a faithful functor this composition stays nonzero. So if the functor preserves coghost maps, coghost index behaves the opposite way from level, that is, it does not decrease, while level does not increase under a change of rings as described in \eqref{eq:levelExactFunctor}. 

We recall the outline of the proof and some intermediate results from \cite{Oppermann/Stovicek:2012}, which we use to prove Theorem~\ref{localGlobalFiniteLevel}. 

Let $\dfpcat{R}$ denote the full subcategory of $\dcat{R}$ of bounded below complexes with degreewise finitely generated homology, that is $\coh_i(X) = 0$ for $i \ll 0$ and $\coh_i(X)$ finitely generated for all $i$. Let $\Prod_+(G)$ be the smallest full subcategory of $\dfpcat{R}$, that contains $G$ and is closed under suspensions and all products that exist in $\dfpcat{R}$. 

\begin{discussion} \label{covariantlyFiniteDplus}
Fix $G$ in $\dfbcat{R}$. Then every object $X$ in $\dfpcat{R}$ has a left $\Prod_+(G)$-ap\-prox\-i\-ma\-tion: There exists a map $m(X)\colon X \to H$ with $H$ in $\Prod_+(G)$, such that any map $X \to H'$ with $H'$ in $\Prod_+(G)$ factors through $m(X)$.

By the proof of the partial converse of the coghost lemma, see for example \cite[1.3]{Beligiannis:2008}, there exists a sequence of $G$-coghost maps associated to left $\Prod_+(G)$-ap\-prox\-i\-ma\-tions:
\[
\begin{tikzcd}[column sep=small]
X = X^0 \ar[rd] & & X^1 \ar[ll,"f^1" swap] \ar[rd] & & X^2 \ar[ll,"f^2" swap] \ar[rd] & & \cdots \nospacepunct{.} \ar[ll,"f^3" swap] \\
 & H^0 \ar[ru,"+1" swap] & & H^1 \ar[ru,"+1" swap] & & H^2 \ar[ru,"+1" swap]
\end{tikzcd}
\]
The horizontal maps are $G$-coghost, and the maps $X^i \to H^i$ are left $\Prod_+(G)$-approximations. Moreover, one has
\begin{equation*} \label{eq:converseCoghostD+}
\level_R^{\Prod_+(G)}(X) = \inf\set{n \geq 0}{(f^1 \circ \cdots \circ f^n) = 0 \text{ in } \dcat{R}} = \cogin_{\dfpcat{R}}^{\Prod_+(G)}(X)\,.
\end{equation*}
\end{discussion}

An object $C$ in an additive category $\cat{C}$ is \emph{cocompact}, if for any family of objects $\cat{X}$ in $\cat{C}$ whose product exists in $\cat{C}$ the natural map
\[
\bigoplus_{X \in \cat{X}} \Hom_\cat{C}(X,C) \to \Hom_\cat{C}\left(\prod_{X \in \cat{X}} X,C\right)
\]
is an isomorphism. 

\begin{discussion} \label{cocompactProd}
By \cite[Theorem~18]{Oppermann/Stovicek:2012}, the cocompact objects in $\dfpcat{R}$ are precisely the bounded complexes, that is the objects in $\dfbcat{R}$. Then given $X$ and $G$ in $\dfbcat{R}$, by \cite[Proposition~2.2.4]{Bondal/VanDenBergh:2003},
\[
\level_R^{\Prod_+(G)}(X) = \level_R^G(X)\,.
\]
\end{discussion}

\begin{discussion} \label{coghostRestrict}
Last, one shows coghost index in $\dfbcat{R}$ is the same as coghost index in $\dfpcat{R}$. Given a nonzero composition of $G$-coghost maps in $\dfpcat{R}$, it is possible to construct a nonzero composition of $G$-coghost maps in $\dfbcat{R}$:

For $G$-coghost maps $f^i\colon X^i \to X^{i-1}$ in $\dfpcat{R}$ for $1 \leq i \leq n$ with $X = X^0$ in $\dfbcat{R}$, there exists a commutative diagram
\[
\begin{tikzcd}
X^n \ar[r,"f^n"] & X^{n-1} \ar[r,"f^{n-1}"] & \cdots \ar[r,"f^2"] & X^1 \ar[r,"f^1"] & X \\
Y^n \ar[r,"g^n"] \ar[u] & Y^{n-1} \ar[r,"g^{n-1}"] \ar[u] & \cdots \ar[r,"g^2"] & Y^1 \ar[r,"g^1"] \ar[u] & X \ar[u,equal]
\end{tikzcd}
\]
where the $Y^i$'s are perfect and the horizontal maps are $G$-coghost. Moreover, the composite map in the top row is zero if and only if the composite map in the bottom row is zero.
\end{discussion}

This concludes the outline of the proof of the converse coghost lemma. 

All steps but \ref{coghostRestrict} of this proof can be adjusted by replacing coghost by ghost maps and $\dfpcat{R}$ by $\dfmcat{R}$. In the last step, the complex $Y^i$ is a truncation of a projective resolution of $X^i$. Here one would need to replace the projective resolution by an injective resolution. But in general, there are no nonzero finitely generated injective modules; so the truncation of an injective resolution need not lie in $\dfbcat{R}$. It seems a converse ghost lemma cannot be proven similarly. However, it is still possible to establish a converse ghost lemma if the ring has a dualizing complex. 

\subsection*{Dualizing complex}\label{subsec:dualizingComplex} 

Let $\mathsf{d} \colon \cat{S} \to \op{\cat{T}}$ and $\mathsf{d}' \colon \cat{S} \to \op{\cat{T}}$ be a duality of triangulated categories in the sense that $\mathsf{d}$ and $\mathsf{d}'$ are contravariant functors, with $\mathsf{d} \mathsf{d}' \cong \id_{\cat{S}}$ and $\mathsf{d}' \mathsf{d} \cong \id_{\cat{T}}$. The duality interchanges ghost and coghost maps, so that
\begin{equation}\label{eq:dualityCoGhost}
\gin_\cat{S}^G(X) = \cogin_{\cat{T}}^{\mathsf{d}(G)}(\mathsf{d}(X)) \quad\text{and}\quad \cogin_\cat{S}^G(X) = \gin_{\cat{T}}^{\mathsf{d}(G)}(\mathsf{d}(X))\,.
\end{equation}
Thus the converse coghost lemma holds for $G$ in $\cat{S}$ if and only if the converse ghost lemma holds for $\mathsf{d}(G)$ in $\cat{T}$. 

The dualizing complex gives a class of dualities on the derived categories. Let $S$ be a left noetherian ring and $R$ a right noetherian ring. For the definition of a dualizing complex $\omega$ of the ordered pair $\langle S,R \rangle$, see \cite[Definition~1.1]{Christensen/Frankild/Holm:2006}. If $R$ is additionally left noetherian, there exists a contravariant auto-equivalence
\[
\begin{tikzcd}[column sep=8em]
\dfbcat{S} \ar[r,"{\RHom_S(-,\omega)}",shift left] & \ar[l,"{\RHom_{\op{R}}(-,\omega)}",shift left] \dfbcat{\op{R}}\nospacepunct{;}
\end{tikzcd}
\]
see \cite[Proposition~3.4]{Iyengar/Krause:2006}. These functors send ghost maps to coghost maps and vice versa. 

\begin{discussion}
Let $S$ be a left noetherian ring and $R$ a commutative noetherian ring with $\omega$ a dualizing complex of $\langle S,R \rangle$. Fix $G$ in $\dfbcat{S}$. By the discussion above and the converse coghost lemma~\ref{converseCoghostLemma}, for any $X$ in $\dfbcat{S}$ we obtain
\[
\gin_{\dfbcat{S}}^G(X) = \level_S^G(X)\,.
\]
\end{discussion}

If $R$ is a commutative noetherian ring, the definition of a dualizing complex of $\langle R,R \rangle$ coincides with Grothendieck's definition of a dualizing complex; see \cite[V \S 2]{Hartshorne:1966}. Then $R$ has a dualizing complex if and only if it is the homomorphic image of a Gorenstein ring of finite Krull dimension; see \cite[Corollary~1.4]{Kawasaki:2002}. So for any such ring, the converse ghost lemma holds. 

\subsection*{Ring maps} Let $\phi \colon R \to S$ be a ring map with $R$ a commutative noetherian ring and $S$ a noetherian ring. Then there are the adjoint functors
\begin{equation}\label{eq:adjunction}
\begin{tikzcd}[column sep=wider]
\dcat{R} \ar[r,shift left,"S \Lotimes_R -"] & \dcat{S} \ar[l,shift left,"{\text{restriction}}"]\nospacepunct{.}
\end{tikzcd}
\end{equation}
If $\phi$ is flat, then $S \Lotimes_R -$ restricts to a functor from $\dfbcat{R}$ to $\dfbcat{S}$. 

\begin{discussion} \label{faithfullyFlatCoghost}
Let $\phi$ be a flat map, $G \in \dfbcat{R}$, and $f$ a $G$-coghost map in $\dfbcat{R}$. Then $\RHom_R(f,G) \simeq 0$. Since $S$ is flat, we have
\begin{equation*}
0 \simeq S \Lotimes_R \RHom_R(f,G) \simeq \RHom_R(f,S \Lotimes_R G)\,.
\end{equation*}
For the second quasi-isomorphism, the proof that the natural map from left to right is quasi-isomorphic is standard. Thus $f$ is $S \Lotimes_R G$-coghost, and, by adjunction, the map $S \Lotimes_R f$ is $S \Lotimes_R G$-coghost in $\dfbcat{S}$.

If additionally $\phi$ is faithful, then $S \Lotimes_R -$ is faithful, and thus for any nonzero map $f$, the map $S \Lotimes_R f$ is nonzero. 
\end{discussion}

The following answers a question posed in \cite[Remarks~9.6]{Dwyer/Greenlees/Iyengar:2006b}. 

\begin{corollary} \label{levelFaithfullyFlat}
Let $\phi\colon R \to S$ be a faithfully flat ring map with $R$ a commutative noetherian ring and $S$ a noetherian ring, so that $R$ acts centrally on $S$. For $X, G \in \dfbcat{R}$, one has $\level_R^G(X) = \level_S^{S \Lotimes_R G}(S \Lotimes_R X)$.
\end{corollary}
\begin{proof}
This follows immediately from the converse coghost lemma~\ref{converseCoghostLemma}
\begin{equation*}
\level_R^G(X) = \cogin^G_{\dfbcat{R}}(X) \leq \cogin^{S \Lotimes_R G}_{\dfbcat{S}}(S \Lotimes_R X) \leq \level^{S \Lotimes_R G}_S(S \Lotimes_R X)\,.
\end{equation*}
The inequality in the middle holds by \ref{faithfullyFlatCoghost}. 
\end{proof}

Note, the proof does not require that the converse coghost lemma holds in $\dfbcat{S}$. 

\begin{corollary} \label{levelCompletion}
Let $R$ be a commutative noetherian ring and let $\complete{(-)}$ be the completion with respect to an ideal $I$ in the Jacobson radical. Then for any $X$, $G$ in $\dfbcat{R}$ we have $\level_R^G(X) = \level_{\complete{R}}^{\complete{G}}(\complete{X})$. \qed
\end{corollary}

\section{A local to global principle}\label{sec:localGlobal}

In this section, we investigate the behavior of level and finite generation in the derived category of a commutative noetherian ring under localization at prime ideals. 

\subsection*{Localization} Let $R$ be a commutative noetherian ring and $\mathfrak{p}$ a prime ideal of $R$. For any left $R$-module $M$ one has $M_\mathfrak{p} \cong R_\mathfrak{p} \otimes_R M$ as a left module over $R_\mathfrak{p}$. The ring map $R \to R_\mathfrak{p}$ is flat, but need not be faithful.

An $R$-module $M$ is zero if and only if $M_\mathfrak{m}$ is zero for all maximal ideals. Thus a map of $R$-modules $f$ is zero if and only if $f_\mathfrak{m} = 0$ for all maximal ideals $\mathfrak{m}$. The same holds for maps in the derived category of complexes with finitely generated total homology:

\begin{lemma} \label{localGlobalFaithful}
Let $f\colon X \to Y$ be a morphism in $\dfbcat{R}$. Then the following conditions are equivalent
\begin{enumerate}
\item\label{localGlobalFaithful:global} $f = 0$ in $\dcat{R}$, 
\item\label{localGlobalFaithful:primes} $f_\mathfrak{p} = 0$ in $\dcat{R_\mathfrak{p}}$ for all $\mathfrak{p} \in \Spec(R)$, and
\item\label{localGlobalFaithful:max} $f_\mathfrak{m} = 0$ in $\dcat{R_\mathfrak{m}}$ for all $\mathfrak{m} \in \Max(R)$.
\end{enumerate}
\end{lemma}
\begin{proof}
\eqref{localGlobalFaithful:global} $\implies$ \eqref{localGlobalFaithful:primes} and \eqref{localGlobalFaithful:primes} $\implies$ \eqref{localGlobalFaithful:max} are obvious. 

For \eqref{localGlobalFaithful:max} $\implies$ \eqref{localGlobalFaithful:global}: Since $X$ and $Y$ lie in $\dfbcat{R}$, we have
\[
\Hom_{\dcat{R}}(X,Y)_\mathfrak{m} = \Hom_{\dcat{R_\mathfrak{m}}}(X_\mathfrak{m},Y_\mathfrak{m})\,.
\]
Now $[f_\mathfrak{m}] = [f]_\mathfrak{m} = 0$ for all maximal ideals $\mathfrak{m}$ if and only if $[f] = 0$. 
\end{proof}

\begin{lemma} \label{mapZeroOpen}
Given the map $f\colon X \to Y$ in $\dfbcat{R}$ the subset
\[
\set{\mathfrak{p} \in \Spec(R)}{f_\mathfrak{p} = 0 \text{ in } \dcat{R_\mathfrak{p}}}
\]
of $\Spec(R)$ is open in the Zariski topology. 
\end{lemma}
\begin{proof}
The map $f$ fits in an exact triangle
\[
X \xrightarrow{f} Y \xrightarrow{g} Z \to \susp X\,.
\]
Applying $\Hom_{\dcat{R}}(X,-)$ to this exact triangle gives the long exact sequence
\[
\dots \to \Hom_{\dcat{R}}(X,X) \xrightarrow{f_*} \Hom_{\dcat{R}}(X,Y) \xrightarrow{g_*} \Hom_{\dcat{R}}(X,Z) \to \cdots\,.
\]
Then $f = 0$ if and only if $\ker(g_*) = 0$. Since $X$ and $Y$ are in $\dfbcat{R}$, one has $\ker(g_*)_\mathfrak{p} = \ker((g_\mathfrak{p})_*)$. Since $\ker(g_*)$ is finitely generated, the set
\[
\set{\mathfrak{p} \in \Spec(R)}{f_\mathfrak{p} = 0 \text{ in } \dcat{R_\mathfrak{p}}} = \set{\mathfrak{p} \in \Spec(R)}{\ker(g_*)_\mathfrak{p} = 0}
\]
is open in the Zariski topology. 
\end{proof}

\begin{lemma} \label{localGlobalCoginUp}
Fix $G$ in $\dfbcat{R}$. Then for any $X \in \dfbcat{R}$, one has
\begin{align*}
\cogin_{\dfbcat{R}}^G(X) \leq& \sup\set{\cogin_{\dfbcat{R_\mathfrak{m}}}^{G_\mathfrak{m}}(X_\mathfrak{m})}{\mathfrak{m} \in \Max(R)} \\
\leq& \sup\set{\cogin_{\dfbcat{R_\mathfrak{p}}}^{G_\mathfrak{p}}(X_\mathfrak{p})}{\mathfrak{p} \in \Spec(R)}\,.
\end{align*}
\end{lemma}
\begin{proof}
Given an $n$-fold $G$-coghost map $f$. Then $f_\mathfrak{m}$ is an $n$-fold $G_\mathfrak{m}$-coghost map by \ref{faithfullyFlatCoghost}. If $f_\mathfrak{m} = 0$ for all maximal ideals $\mathfrak{m}$, then $f = 0$ by Lemma~\ref{localGlobalFaithful}. This proves the first inequality. The second is obvious. 
\end{proof}

\subsection*{Local to global principle} Using the preceding lemma, the following is a consequence of the converse coghost lemma~\ref{converseCoghostLemma}. 

\begin{corollary} \label{levelLocalGlobal}
Let $R$ be a commutative noetherian ring. Fix $G$ and $X$ in $\dfbcat{R}$. Then
\begin{align*}
\level_R^G(X) =& \sup\set{\level_{R_\mathfrak{p}}^{G_\mathfrak{p}}(X_\mathfrak{p})}{\mathfrak{p} \in \Spec(R)} \\
=& \sup\set{\level_{R_\mathfrak{m}}^{G_\mathfrak{m}}(X_\mathfrak{m})}{\mathfrak{m} \in \Max(R)}\,.
\end{align*}
\end{corollary}
\begin{proof}
Given a prime ideal $\mathfrak{p}$, there exists a maximal ideal $\mathfrak{m} \supset \mathfrak{p}$ and, by \eqref{eq:levelExactFunctor},
\[
\level_R^G(X) \geq \level_{R_\mathfrak{m}}^{G_\mathfrak{m}}(X_\mathfrak{m}) \geq \level_{R_\mathfrak{p}}^{G_\mathfrak{p}}(X_\mathfrak{p})\,.
\]
So it is enough to show the claim for all maximal ideals. By the converse coghost lemma~\ref{converseCoghostLemma} and Lemma~\ref{localGlobalCoginUp} one has
\begin{align*}
\level_R^G(X) =& \cogin_{\dfbcat{R}}^G(X) \\
\leq & \sup\set{\cogin_{\dfbcat{R_\mathfrak{m}}}^{G_\mathfrak{m}}(X_\mathfrak{m})}{\mathfrak{m} \in \Max(R)} \\
=& \sup\set{\level_{R_\mathfrak{m}}^{G_\mathfrak{m}}(X_\mathfrak{m})}{\mathfrak{m} \in \Max(R)}\,.
\end{align*}
For the opposite inequality, 
\[
\level_R^G(X) \geq \level_{R_\mathfrak{m}}^{G_\mathfrak{m}}(X_\mathfrak{m})
\]
holds for all maximal ideals $\mathfrak{m} \in \Max(R)$, by \eqref{eq:levelExactFunctor}.
\end{proof}

In \cite[Lemma~4.5]{Bass/Murthy:1967} it is proved that a module $M$ has finite projective dimension if and only if $M_\mathfrak{p}$ has finite projective dimension for all prime ideals $\mathfrak{p}$. This was extended to perfect complexes by \cite[Theorem~4.1]{Avramov/Iyengar/Lipman:2010}. Next we generalize these results to $\level$ with respect to any $G$. This extends Corollary~\ref{levelLocalGlobal}, in that it is not only possible to compute level locally, but also to check finiteness of level locally. For example, when a rings has infinitely many maximal ideals this is not a consequence of Corollary~\ref{levelLocalGlobal}. 

\begin{proposition} \label{primesLevelOpen}
Let $R$ be a commutative noetherian ring. Suppose $G$ and $X$ are objects in $\dfbcat{R}$. Then for any integer $n$ the set
\[
\mathcal{V}_n \colonequals \set{\mathfrak{p} \in \Spec(R)}{\level_{R_\mathfrak{p}}^{G_\mathfrak{p}}(X_\mathfrak{p}) \leq n} \subset \Spec(R)
\]
is Zariski open.
\end{proposition}
\begin{proof}
We use the intermediate results in the proof of the converse coghost lemma \cite[Theorem~24]{Oppermann/Stovicek:2012}, recalled in Section~\ref{sec:levelCoghost}. 

By \ref{covariantlyFiniteDplus}, there exists a sequence of $G$-coghost maps induced by the left $\Prod_+(G)$-approximations
\[
f^{i+1}\colon X^{i+1} \to X^i \quad \text{in }\dfpcat{R}
\]
with $X^0 = X$, such that
\[
\level_R^G(X) = \inf\set{n \geq 0}{(f^1 \circ \dots \circ f^n) = 0 \text{ in } \dcat{R}}\,.
\]
Then, by \ref{coghostRestrict}, there exist $G$-coghost maps $g^i\colon Y^i \to Y^{i-1}$ with $Y^0 = X$ and $Y^i$ perfect, such that $g^1 \circ \dots \circ g^n$ is zero if and only if $f^1 \circ \dots \circ f^n$ is zero. That gives
\[
\level_R^G(X) = \inf\set{n \geq 0}{(g^1 \circ \dots \circ g^n) = 0 \text{ in } \dcat{R}}\,.
\]

Now we want to show that the construction of the $g^i$'s from $X$ localizes. Specifically, we require that $\Prod_+(G)$, the left $\Prod_+(G)$-approximations and the descent from the $f^i$'s to the $g^i$'s in \ref{coghostRestrict} localize.

While products need not localize in general, the products in $\dfpcat{R}$ localize: If a product exists in $\dfpcat{R}$, then by \cite[Proposition~13]{Oppermann/Stovicek:2012} it is a componentwise product and one may assume that each component in the product is finite. Thus $\Prod_+(G_\mathfrak{p}) = \Prod_+(G)_\mathfrak{p}$. 

Let $m(Z) \colon Z \to H$ be a left $\Prod_+(G)$-approximation for some $Z \in \dfpcat{R}$. Completing this map to an exact triangle $W \to Z \to H \to \susp W$, yields a $G$-coghost map $h \colon W \to Z$. By \ref{faithfullyFlatCoghost}, its localization $h_\mathfrak{p}$ is $G_\mathfrak{p}$-ghost. Since $H_\mathfrak{p}$ lies in $\Prod_+(G_\mathfrak{p})$, the localization $m(Z_\mathfrak{p})$ is a left $\Prod_+(G_\mathfrak{p})$-approximation.

Lastly, the construction \ref{coghostRestrict} descends to the localization, so that $(f^1 \circ \cdots \circ f^n)_\mathfrak{p}$ is zero if and only if $(g^1 \circ \cdots \circ g^n)_\mathfrak{p}$ is zero. So
\[
\mathcal{V}_n = \set{\mathfrak{p} \in \Spec(R)}{(g^1 \circ \dots \circ g^n)_\mathfrak{p} = 0 \text{ in } \dcat{R_\mathfrak{p}}}
\]
is open by Lemma~\ref{mapZeroOpen}. 
\end{proof}

\begin{theorem} \label{localGlobalFiniteLevel}
Let $R$ be a commutative noetherian ring. For objects $G$ and $X$ in $\dfbcat{R}$, the following conditions are equivalent
\begin{enumerate}
\item\label{localGlobalFiniteLevel:global} $\level_R^G(X) < \infty$, 
\item\label{localGlobalFiniteLevel:primes} $\level_{R_\mathfrak{p}}^{G_\mathfrak{p}}(X_\mathfrak{p}) < \infty$ for all $\mathfrak{p} \in \Spec(R)$, and
\item\label{localGlobalFiniteLevel:max} $\level_{R_\mathfrak{m}}^{G_\mathfrak{m}}(X_\mathfrak{m}) < \infty$ for all $\mathfrak{m} \in \Max(R)$.
\end{enumerate}
\end{theorem}
\begin{proof}
The implications \eqref{localGlobalFiniteLevel:global} $\implies$ \eqref{localGlobalFiniteLevel:primes} and \eqref{localGlobalFiniteLevel:primes} $\iff$ \eqref{localGlobalFiniteLevel:max} are clear. For \eqref{localGlobalFiniteLevel:primes} $\implies$ \eqref{localGlobalFiniteLevel:global}, assume $\level_{R_\mathfrak{p}}^{G_\mathfrak{p}}(X_\mathfrak{p})$ is finite for all prime ideals $\mathfrak{p} \in \Spec(R)$. That is, the union of all $\mathcal{V}_n$, using the notation of the proof of Proposition~\ref{primesLevelOpen}, is $\Spec(R)$. By Proposition~\ref{primesLevelOpen}, the $\mathcal{V}_n$'s form an ascending chain of open sets. Since $R$ is noetherian, the space $\Spec(R)$ is noetherian, and the chain stabilizes. So there exists an $N$, such that $\mathcal{V}_N = \Spec(R)$. Thus $\level_{R_\mathfrak{p}}^{G_\mathfrak{p}}(X_\mathfrak{p}) \leq N$ for all prime ideals $\mathfrak{p}$, and by Corollary~\ref{levelLocalGlobal}, $\level_R^G(X) < \infty$. 
\end{proof}

\begin{discussion} \label{NoetherAlgebra}
A ring $R$ is a \emph{Noether algebra} if its center $\ctr(R)$ is noetherian and $R$ is a finitely generated module over $\ctr(R)$. 

For example, given a finitely generated module $M$ over a commutative noetherian ring $A$, the endomorphism ring $\Hom_A(M,M)$ is a Noether algebra, where the image of $A$ lies in the center. 

Theorem~\ref{localGlobalFiniteLevel} and all the previous results in this section hold when $R$ is a Noether algebra, when one localizes with respect to the prime ideals of $\ctr(R)$. This is, because the converse coghost lemma by \cite[Theorem~24]{Oppermann/Stovicek:2012} holds for Noether algebras. 
\end{discussion}

\subsection*{Local to global principle for upper bounds} One way to think about level of $X$ with respect to $G$ is as the generation time for $X$ when using $G$ as a building block. It is interesting to know whether $G$ generates every object and if there is an upper bound for level of any objects in a given triangulated category. 

\begin{definition}
An object $G$ in $\cat{T}$ is a \emph{strong generator} of $\cat{T}$ if $\thick_\cat{T}^n(G) = \cat{T}$ for some $n$. The \emph{generation time} of $G$ is defined by
\[
\gentime_\cat{T}(G) = \inf\set{n \geq 0}{\thick_\cat{T}^{n+1}(G) = \cat{T}}\,.
\]
\end{definition}

The notation is adopted from \cite[Definition~2.1]{Ballard/Favero/Katzarkov:2012}. The generation time is shifted by one from level. That is
\[
\level_\cat{T}^G(X) \leq \gentime_\cat{T}(G) + 1\,.
\]
The convention for the generation time is so that it matches up with the Rouquier dimension; see \cite[Definition~3.2]{Rouquier:2008}. 

To detect whether an object is a strong generator locally, one has to be able to lift objects from the localizations. 

\begin{lemma} \label{localizationEssentiallySurj}
For any prime ideal $\mathfrak{p} \in \Spec(R)$, the functor $\dfbcat{R} \to \dfbcat{R_\mathfrak{p}}$ is essentially surjective, that is for every $X \in \dfbcat{R_\mathfrak{p}}$ there exists $Y \in \dfbcat{R}$ such that $Y_\mathfrak{p} \simeq X$. 
\end{lemma}
\begin{proof}
Every finitely generated module over $R_\mathfrak{p}$ can be lifted to a finitely generated module over $R$. For finitely generated $R$-modules $M$ and $N$ one has
\begin{equation*}
\Hom_R(M,N) \otimes_R R_\mathfrak{p} \cong \Hom_{R_\mathfrak{p}}(M_\mathfrak{p},N_\mathfrak{p})\,.
\end{equation*}
Thus for a map $f \colon M_\mathfrak{p} \to N_\mathfrak{p}$ there exists $r \in R \setminus \mathfrak{p}$ such that $r \cdot f = g_\mathfrak{p}$ where $g \colon M \to N$ a map over $R$. 

Let $X$ be a complex, and $X'$ the complex with the same modules, and differentials $r_i \cdot \partial^X_i$ where the $r_i$'s are units, and $r_i = 1$ for all but finitely many integers $i$. Then $X$ and $X'$ are quasi-isomorphic. 

Given a sequence
\[
X \xrightarrow{f} Y \xrightarrow{g} Z
\]
of finitely generated modules over $R_\mathfrak{p}$ with $g \circ f = 0$. Up to multiplication by elements in $R \setminus \mathfrak{p}$, this sequence lifts to a sequence
\[
\tilde{X} \xrightarrow{\tilde{f}} \tilde{Y} \xrightarrow{\tilde{g}} \tilde{Z}
\]
over $R$. It is not necessary that $\tilde{g} \circ \tilde{f} = 0$, but one has $(\tilde{g} \circ \tilde{f})_\mathfrak{p} = 0$. Since $X$ is finitely generated there exists $r \in R \setminus \mathfrak{p}$, such that $r \cdot (\tilde{g} \circ \tilde{f})(X) = 0$. Replacing $\tilde{g}$ by $r \tilde{g}$ gives a sequence whose composition is zero. Since $r$ is a unit in $R_\mathfrak{p}$ this sequence localizes to the original sequence. Thus inductively any bounded complex of finitely generated $R_\mathfrak{p}$-modules lifts to a complex of finitely generated $R$-modules. 
\end{proof}

It is possible to detect a strong generator locally, when the generation time has an upper bound locally. 

\begin{theorem} \label{localGlobalStrongGenerator}
Let $R$ be a commutative noetherian ring. Fix $G$ in $\dfbcat{R}$ and a positive integer $N$. Then the following are equivalent
\begin{enumerate}
\item\label{localGlobalStrongGenerator:global} $G$ is a strong generator of $\dfbcat{R}$ with $\gentime_{\dfbcat{R}}(G) \leq N$, 
\item\label{localGlobalStrongGenerator:primes} $G_\mathfrak{p}$ is a strong generator of $\dfbcat{R_\mathfrak{p}}$ with $\gentime_{\dfbcat{R_\mathfrak{p}}}(G_\mathfrak{p}) \leq N$ for all prime ideals $\mathfrak{p} \in \Spec(R)$, and
\item\label{localGlobalStrongGenerator:max} $G_\mathfrak{m}$ is a strong generator of $\dfbcat{R_\mathfrak{m}}$ with $\gentime_{\dfbcat{R_\mathfrak{m}}}(G_\mathfrak{m}) \leq N$ for all maximal ideals $\mathfrak{m} \in \Max(R)$.
\end{enumerate}
\end{theorem}
\begin{proof}
\eqref{localGlobalStrongGenerator:primes} $\implies$ \eqref{localGlobalStrongGenerator:max} is obvious. For \eqref{localGlobalStrongGenerator:global} $\implies$ \eqref{localGlobalStrongGenerator:primes}: Given any $X$ in $\dfbcat{R_\mathfrak{p}}$, by Lemma~\ref{localizationEssentiallySurj}, there exists $Y$ in $\dfbcat{R}$ with $Y_\mathfrak{p} = X$. One has
\[
\level_{R_\mathfrak{p}}^{G_\mathfrak{p}}(X) \leq \level_R^G(Y) \leq \gentime_{\dfbcat{R}}(G) + 1 \leq N + 1\,.
\]
So $G_\mathfrak{p}$ is a strong generator of $\dfbcat{R_\mathfrak{p}}$ with generation time $\leq N$.

It remains to show \eqref{localGlobalStrongGenerator:max} $\implies$ \eqref{localGlobalStrongGenerator:global}. For any $X$ in $\dfbcat{R}$, we have, by Corollary~\ref{levelLocalGlobal},
\begin{align*}
\level_R^G(X) =& \sup\set{\level_{R_\mathfrak{m}}^{G_\mathfrak{m}}(X_\mathfrak{m})}{\mathfrak{m} \in \Max(R)} \\
\leq& \sup\set{\gentime_{\dfbcat{R_\mathfrak{m}}}(G_\mathfrak{m})}{\mathfrak{m} \in \Max(R)} + 1 \leq N+1\,,
\end{align*}
and so $G$ is a strong generator of $\dfbcat{R}$ with $\gentime_{\dfbcat{R}}(G) \leq N$. 
\end{proof}

This statement does not hold without a uniform bound on local generation time. 

\begin{example}
In \cite[Appendix~A1]{Nagata:1962} Nagata constructed a commutative noetherian ring $R$ of infinite Krull dimension, such that $R_\mathfrak{m}$ is regular and of finite Krull dimension for all maximal ideals $\mathfrak{m}$. So
\[
\gentime_{\dfbcat{R_\mathfrak{m}}}(R_\mathfrak{m}) = \gldim(R_\mathfrak{m}) = \dim(R_\mathfrak{m}) < \infty
\]
for any maximal ideal $\mathfrak{m}$. But
\[
\gentime_{\dfbcat{R}}(R) = \gldim(R) = \dim(R) = \infty\,.
\]
So $R_\mathfrak{m}$ is a strong generator of $\dfbcat{R_\mathfrak{m}}$ for all $\mathfrak{m}$, but $R$ is not a strong generator of $\dfbcat{R}$. 
\end{example}

\section{Applications}\label{sec:application}

Combining Corollaries~\ref{levelCompletion}, \ref{levelLocalGlobal} and Theorem~\ref{localGlobalFiniteLevel} we obtain:

\begin{corollary} \label{reductionCompleteLocal}
Given a commutative noetherian ring $R$, and $G$, $X$ in $\dfbcat{R}$, one has
\[
\level_R^G(X) = \sup\set{\level_{\complete{R_\mathfrak{m}}}^{\complete{G_\mathfrak{m}}}(\complete{X_\mathfrak{m}})}{\mathfrak{m} \in \Max(R)}
\]
and $\level_R^G(X) < \infty$ if and only if $\level_{\complete{R_\mathfrak{m}}}^{\complete{G_\mathfrak{m}}}(\complete{X_\mathfrak{m}}) < \infty$ for all maximal ideals $\mathfrak{m}$. Here $\complete{(-)}$ denotes the completion in $R_\mathfrak{m}$ with respect to its maximal ideal $\mathfrak{m} R_\mathfrak{m}$. 
\qed
\end{corollary}

\subsection*{Theorem of Hopkins and Neeman for complexes of finite injective dimension} For a complex $X$, let
\[
\supp_R(X) \colonequals \set{\mathfrak{p} \in \Spec(R)}{\kappa(\mathfrak{p}) \Lotimes_R X \not\simeq 0}
\]
be the \emph{support of $X$}, where $\kappa(\mathfrak{p}) \colonequals R_\mathfrak{p}/\mathfrak{p} R_\mathfrak{p}$ is the residue field of $R_\mathfrak{p}$. If $X$ lies in $\dfbcat{R}$, this coincides with the support of its homology $\Supp_R(\coh(X))$; see the discussion preceding \cite[Lemma~2.6]{Foxby:1979}. In particular, the support $\supp_R(X)$ is closed for $X$ in $\dfbcat{R}$. 

\begin{discussion} \label{perfectCxSupportGeneration}
For perfect complexes $X$ and $Y$, Hopkins \cite[Theorem~11]{Hopkins:1987} and Neeman \cite[Lemma~1.2]{Neeman:1992b} prove that if $\supp_R(X) \subset \supp_R(Y)$, then $\level_R^Y(X) < \infty$. In particular, this gives a simple condition when a thick subcategory generated by $Y$ contains the thick subcategory generated by $X$.
\end{discussion}

\begin{discussion} \label{dualizingPerfectFinInjDim}
If $R$ has a dualizing complex $\omega$, as introduced in Section~\ref{sec:levelCoghost}, one gets an equivalence of categories
\[
\begin{tikzcd}[column sep=8em]
\Perf{R} \ar[r,"{\RHom_R(-,\omega)}",shift left] & \ar[l,"{\RHom_R(-,\omega)}",shift left] \kbfcat{\Inj{R}} \nospacepunct{,}
\end{tikzcd}
\]
where $\kbfcat{\Inj{R}}$ is the homotopy category of complexes that are quasi-isomorphic to a bounded complexes of injective $R$-modules with finitely generated homology; see \cite[Section~2.3]{Roberts:1980}. 
\end{discussion}

When $R$ has a dualizing complex, then an analogue of \ref{perfectCxSupportGeneration} holds for complexes in $\kbfcat{\Inj{R}}$. Using that every complete local ring has a dualizing complex, we prove a more general analogue. 

\begin{theorem} \label{finInjDimCxSupportGeneration}
Let $R$ be a commutative noetherian ring, and let $X$ and $Y$ be objects in $\kbfcat{\Inj{R}}$ with $\supp_R(X) \subset \supp_R(Y)$. Then $\level_R^Y(X) < \infty$.
\end{theorem}
\begin{proof}
By \ref{reductionCompleteLocal}, the conclusion holds if and only if it holds in $\dfbcat{\complete{R_\mathfrak{m}}}$ for all maximal ideals $\mathfrak{m}$. We will show the assumptions also descend to $\dfbcat{\complete{R_\mathfrak{m}}}$.

Let $\mathfrak{m}$ be any maximal ideal. Then $X_\mathfrak{m}$ and $Y_\mathfrak{m}$ lie in $\kbfcat{\Inj{R_\mathfrak{m}}}$. Since localization also preserves the inclusion of their support, we may assume $R$ is local. 

Let $\complete{(-)}$ denote the completion with respect to the maximal ideal and $k$ the residue field of $R$. By \cite[Proposition~5.5(I)]{Avramov/Foxby:1991}, $X \in \kbfcat{\Inj{R}}$ if and only if $\RHom_R(k,X)$ is a bounded above complex. Since $X \in \kbfcat{\Inj{R}}$, it lies in particular in $\dfbcat{R}$, so that $\complete{X} \cong X \Lotimes_R \complete{R}$. Then
\[
\RHom_{\complete{R}}(k,\complete{X}) \cong \RHom_R(k,\complete{X}) \cong \RHom_R(k,X) \Lotimes_R \complete{R}
\]
and thus $X \in \kbfcat{\Inj{R}}$ if and only if $\complete{X} \in \kbfcat{\Inj{\complete{R}}}$. 

It is well known, that
\[
({}^a\phi)^{-1}(\supp_R(X)) = ({}^a\phi)^{-1}(\Supp_R(\coh(X))) = \Supp_{\complete{R}}(\complete{R} \otimes_R \coh(X)) = \supp_{\hat{R}}(\complete{X})
\]
where $\phi \colon R \to \complete{R}$ is the canonical ring homomorphism and ${}^a\phi \colon \Spec(\complete{R}) \to \Spec(R)$ the induced map. So completion preserves the inclusion of the support. 

Thus without loss of generality we assume $R$ is a complete local ring. Now $R$ has a dualizing complex $\omega$, and the claim holds.
\end{proof}

\section{Virtual and proxy smallness}\label{sec:VirtualProxySmallness}

In the derived category $\dcat{R}$ of a noetherian ring, the perfect complexes are precisely the compact---also called small---objects. That is the perfect complexes are precisely the complexes $P$ for which the functor 
\[
\RHom_R(P,-) \colon \dcat{R} \to \dcat{R}
\]
commutes with direct sums. There are two notions on how to describe complexes that are almost small; see \cite[4.1]{Dwyer/Greenlees/Iyengar:2006b}. 

\begin{definition} 
A complex $X$ in $\dcat{R}$ is \emph{virtually small}, if $X \simeq 0$ or there exists $P \not\simeq 0$ in $\dcat{R}$, such that
\begin{equation} \label{eq:virtuallySmall}
\level_R^R(P) < \infty \quad\text{and}\quad \level_R^X(P) < \infty\,.
\end{equation}
If additionally $\supp_R(X) = \supp_R(P)$, then $X$ is \emph{proxy small}. 
\end{definition}

\begin{discussion} \label{VirtProxyKoszulCx}
By \cite[Proposition~4.5]{Dwyer/Greenlees/Iyengar:2006b}, a nonzero complex $X$ is virtually small if and only if \eqref{eq:virtuallySmall} holds for $P = \Kos(\bm{x})$ the Koszul complex on a generating set $\bm{x}$ of some maximal ideal $\mathfrak{m}$. 

Similarly, by \cite[Proposition~4.4]{Dwyer/Greenlees/Iyengar:2006b}, a complex $X$ is proxy small if and only if \eqref{eq:virtuallySmall} holds for $P = \Kos(\bm{x})$ the Koszul complex on any, or equivalently all, sets $\bm{x}$ with $V(\bm{x}) = \supp_R(X)$. 
\end{discussion}

Motivated by Theorem~\ref{localGlobalFiniteLevel}, we track the behavior of proxy smallness under localization. 

\begin{proposition}\label{localGlobalProxySmall}
Let $R$ be a commutative noetherian ring and $X$ in $\dfbcat{R}$. Then $X$ is proxy small if and only if $X_\mathfrak{p}$ is proxy small for all $\mathfrak{p} \in \Spec(R)$. 
\end{proposition}
\begin{proof}
For the if direction: Let $\bm{x} \subset R$ be a set of elements, such that $V(\bm{x}) = \supp_R(X)$. By \ref{VirtProxyKoszulCx}, it is enough to show \eqref{eq:virtuallySmall} for $P = K(\bm{x})$. For any prime ideal $\mathfrak{p}$, we have $\Kos(\bm{x})_\mathfrak{p} \simeq \Kos(\bm{x}_\mathfrak{p})$. Since $X_\mathfrak{p}$ is proxy small, $\level_{R_\mathfrak{p}}^{X_\mathfrak{p}}(K(\bm{x})_\mathfrak{p}) < \infty$ for all prime ideals $\mathfrak{p}$. So by Theorem~\ref{localGlobalFiniteLevel}, the complex $X$ is proxy small. 

For the only if direction, let $P$ be a perfect complex, such that 
\[
\level_R^X(P) < \infty \quad\text{and}\quad \supp_R(X) = \supp_R(P)\,.
\]
Then for any $\mathfrak{p} \in \supp_R(X)$, one has that $P_\mathfrak{p} \not\simeq 0$ is perfect, and $\level_{R_\mathfrak{p}}^{X_\mathfrak{p}}(P_\mathfrak{p}) < \infty$. If $\mathfrak{p} \notin \supp_R(X)$, then $X_\mathfrak{p} \simeq 0$. Thus $X_\mathfrak{p}$ is proxy small for any prime ideal $\mathfrak{p}$. 
\end{proof}

Virtual smallness does not behave in the same way: The localization of a virtually small complex need not be virtually small. The description of virtual smallness in \ref{VirtProxyKoszulCx} indicates that instead of localizing at all primes, one need only localize at one maximal ideal. 

\begin{proposition}\label{localGlobalVirtuallySmall}
Let $R$ be a commutative noetherian ring and $X \not\simeq 0$ a complex over $R$. Then $X_\mathfrak{m} \not\simeq 0$ is virtually small for some maximal ideal $\mathfrak{m}$ if and only if $X$ is virtually small. 
\end{proposition}
\begin{proof}
For the if direction: Assume $X$ is virtually small. By \ref{VirtProxyKoszulCx}, there exists a maximal ideal $\mathfrak{m}$ and a generating set $\bm{x}$ of $\mathfrak{m}$, such that $\level_R^X(\Kos(\bm{x})) < \infty$. Then $\Kos(\bm{x})_\mathfrak{m} \not\simeq 0$ is a perfect complex and $\level_{R_\mathfrak{m}}^{X_\mathfrak{m}}(\Kos(\bm{x})_\mathfrak{m}) < \infty$. Thus $X_\mathfrak{m}$ is virtually small.

For the only if direction: By hypothesis, there exists $\mathfrak{m} \in \supp_R(X)$, such that $X_\mathfrak{m}$ is virtually small and thus, by \ref{VirtProxyKoszulCx}, $\level_{R_\mathfrak{m}}^{X_\mathfrak{m}}(\Kos(\bm{x})_\mathfrak{m}) < \infty$ for some generating set $\bm{x}$ of $\mathfrak{m}$. For any prime ideal $\mathfrak{p} \neq \mathfrak{m}$ one has $\Kos(\bm{x})_\mathfrak{p} \simeq 0$. So by Theorem~\ref{localGlobalFiniteLevel}, $\level_R^X(\Kos(\bm{x})) < \infty$. By \ref{VirtProxyKoszulCx}, $X$ is virtually small.
\end{proof}

We can also track the behavior of virtual and proxy smallness under a faithfully flat ring map.

\begin{proposition}\label{faithfullyFlatSmall}
Let $\phi \colon R \to S$ be a faithfully flat ring map of commutative noetherian rings and $X \in \dfbcat{R}$. 
\begin{enumerate}
\item\label{faithfullyFlatSmall:proxy} $X$ is proxy small if and only if $S \Lotimes_R X$ is proxy small in $\dcat{S}$.
\item\label{faithfullyFlatSmall:virtually} If $X$ is virtually small, then $S \Lotimes_R X$ is virtually small in $\dcat{R}$.
\end{enumerate}
\end{proposition}
\begin{proof}
Since the functor $S \Lotimes_R -$ is faithful, $X \simeq 0$ if and only if $S \Lotimes_R X \simeq 0$. Thus we may assume $X \not\simeq 0$. Let $\bm{x}$ be a set of elements in $R$, such that $V(\bm{x}) = \supp_R(X)$, and set $I \colonequals (\bm{x})$. Given that $S$ is faithfully flat over $R$, it is well known that
\[
S \Lotimes_R \Kos(\bm{x}) \cong \Kos(\bm{y}) \quad\text{and}\quad \supp_S(S \Lotimes_R X) = V(S \otimes_R I)
\]
for some generating set $\bm{y}$ of $S \otimes_R I$. Then, by \ref{levelFaithfullyFlat}, one has
\[
\level_R^X(\Kos(\bm{x})) = \level_S^{S \Lotimes_R X}(\Kos(\bm{y}))\,.
\]
By \ref{VirtProxyKoszulCx}, the complex $X$ is proxy small if and only if $\level_R^X(\Kos(\bm{x})) < \infty$, and $S \Lotimes_R X$ is proxy small if and only if $\level_S^{S \Lotimes_R X}(\Kos(\bm{y})) < \infty$. This shows \eqref{faithfullyFlatSmall:proxy}. 

For \eqref{faithfullyFlatSmall:virtually}, let $P \not\simeq 0$ be a perfect complex, such that $\level_R^X(P) < \infty$. By \ref{faithfullyFlatCoghost}, the functor $S \Lotimes_R -$ is faithful, so $S \Lotimes_R P \not\simeq 0$ and by \eqref{eq:levelExactFunctor}
\[
\level_S^{S \Lotimes_R X}(S \Lotimes_R P) \leq \level_R^X(P) < \infty\,.
\]
So $S \Lotimes_R X$ is virtually small. 
\end{proof}

The properties virtual and proxy smallness can be used to give a categorical description of complete intersection rings. A local ring $(R,\mathfrak{m},k)$ is \emph{complete intersection}, if its $\mathfrak{m}$-adic completion $\complete{R}$ is of the form $\complete{R} = Q/(f_1, \dots, f_c)$ where $Q$ is a regular local ring and $f_1, \dots, f_c$ a regular sequence in $Q$. 

A commutative noetherian ring $R$ is \emph{locally complete intersection} if for any prime ideal $\mathfrak{p}$ the ring $R_\mathfrak{p}$ is complete intersection. As a consequence of \cite[Theorem~5.4]{Pollitz:2019} and Proposition~\ref{localGlobalProxySmall} we get a characterization of locally complete intersection rings. 

\begin{corollary} \label{locallyCIProxySmall}
For a commutative noetherian ring $R$ the following are equivalent
\begin{enumerate}
\item $R$ is locally complete intersection, and
\item every object in $\dfbcat{R}$ is proxy small.
\end{enumerate}
\end{corollary}
\begin{proof}
Assume $R$ is locally complete intersection. That is $\complete{R_\mathfrak{p}}$ is a quotient of a regular local ring by an ideal generated by a regular sequence. By \cite[Theorem~9.4]{Dwyer/Greenlees/Iyengar:2006b}, every object in $\dfbcat{\complete{R_\mathfrak{p}}}$ is proxy small. Then, by Propositions~\ref{localGlobalProxySmall} and \ref{faithfullyFlatSmall}, every object in $\dfbcat{R}$ is proxy small. 

For the opposite direction, by Lemma~\ref{localizationEssentiallySurj}, the functor $\dfbcat{R} \to \dfbcat{R_\mathfrak{p}}$ is essentially surjective and thus since every object in $\dfbcat{R}$ is proxy small, so is every object in $\dfbcat{R_\mathfrak{p}}$. Then, by \cite[Theorem~5.2]{Pollitz:2019}, the ring $R_\mathfrak{p}$ is complete intersection. 
\end{proof}

In \cite[Theorem~5.4]{Pollitz:2019} Pollitz proved that {\it (1)} holds if and only if every object in $\dfbcat{R}$ is virtually small.

Over a local ring $(R,\mathfrak{m},k)$ a complex $X \in \dfbcat{R}$ has \emph{finite CI-dimension}, if there exist local homomorphisms $R \to R' \gets Q$, such that
\begin{itemize}
\item $R \to R'$ is faithfully flat, 
\item $Q \to R'$ is surjective and the kernel is generated by a regular sequence, and
\item $\level^Q_Q(R' \Lotimes_R X) < \infty$. 
\end{itemize}
This was first introduced by \cite[Section~1]{Avramov/Gasharov/Peeva:1997} and extended to complexes by \cite[Section~3]{SatherWagstaff:2004}. 

Corollary~\ref{levelFaithfullyFlat} answers the question raised in \cite[Remarks~9.6]{Dwyer/Greenlees/Iyengar:2006b}. So we can complete the proof that a complex of finite CI-dimension is virtually small. This has been proven using a different method by \cite[Corollary~3.3]{Bergh:2009}. Using Proposition~\ref{faithfullyFlatSmall} we can strengthen the result to the following. 

\begin{proposition} \label{finiteCIProxySmall}
Every complex in $\dfbcat{R}$ of finite CI-dimension is proxy small. 
\end{proposition}
\begin{proof}
Let $X$ be a complex in $\dfbcat{R}$ of finite CI-dimension and let $R \to R' \gets Q$ be a diagram of local homomorphisms satsifying the required conditions. Then $R' \Lotimes_R X$ has finite homology over $R'$ and in particular over $Q$. So $R' \Lotimes_R X$ is a perfect complex over $Q$. Then, by \cite[Theorem~9.1]{Dwyer/Greenlees/Iyengar:2006b}, the complex $R' \Lotimes_R X$ is proxy small over $R'$ and, by Proposition~\ref{faithfullyFlatSmall}~\eqref{faithfullyFlatSmall:proxy}, $X$ is proxy small in $\dcat{R}$. 
\end{proof}

\begin{discussion}
The definition of finite CI-dimension can be extended to a nonlocal ring $R$; for the definition see \cite[Definition~3.1]{SatherWagstaff:2004}. In particular, when a complex $X$ in $\dfbcat{R}$ has finite CI-dimension over $R$, then $X_\mathfrak{m}$ has finite CI-dimension over $R_\mathfrak{m}$ for all maximal ideals $\mathfrak{m}$. By Proposition~\ref{localGlobalProxySmall}, we can conclude that Proposition~\ref{finiteCIProxySmall} holds over nonlocal rings. 
\end{discussion}

The condition given in Corollary~\ref{locallyCIProxySmall} to test whether a ring is locally complete intersection, is difficult to use: It is hard to check whether every bounded complex with finite homology is proxy small. For local rings, the proof of \cite[Theorem~5.2]{Pollitz:2019} shows it is enough to test finitely many complexes of finite length homology for proxy smallness. The next theorem shows that for some rings it is enough to check one object for proxy smallness. 

Given a $k$-algebra $R$, the enveloping algebra of $R$ is $\env{R} \colonequals R \otimes_k R$. Then $\env{R}$ acts on $R$ diagonally. 

\begin{theorem} \label{kAlgebraLCI}
Let $k$ be a field and $R$ a $k$-algebra essentially of finite type over $k$. Then the following are equivalent
\begin{enumerate}
\item $R$ is locally complete intersection, and
\item $R$ is proxy small in $\dcat{\env{R}}$.
\end{enumerate}
\end{theorem}
\begin{proof}
If $R$ is locally complete intersection, then its localizations $R_\mathfrak{p}$ are complete intersection. So, by \cite{Avramov:1975}, the enveloping algebra $\env{(R_\mathfrak{p})}$ is complete intersection. By Corollary~\ref{locallyCIProxySmall}, every object in $\dfbcat{\env{(R_\mathfrak{p})}}$ is proxy small and thus $R_\mathfrak{p}$ is proxy small in $\dcat{\env{(R_\mathfrak{p})}}$. 

It remains to show $R$ is proxy small in $\dcat{\env{R}}$. By Proposition~\ref{localGlobalProxySmall}, it is enough to show $R_\mathfrak{q}$ is proxy small in $\dcat{(\env{R})_\mathfrak{q}}$ for all prime ideals $\mathfrak{q}$ of $\env{R}$. 

Let $\mu \colon \env{R} \to R$ be the multiplication map, and $\mu^a$ the induced map on spectra. For a prime ideal $\mathfrak{q}$ of $\env{R}$, we have $R_\mathfrak{q} = R_{(\mathfrak{q},\ker(\mu))}$. Now there are two cases, either $(\mathfrak{q},\ker(\mu)) = \env{R}$, or $(\mathfrak{q},\ker(\mu)) = \mu^a(\mathfrak{p})$ for some prime ideal $\mathfrak{p}$ of $R$. In the first case, one has $R_\mathfrak{q} = 0$ and there is nothing to prove. In the second case, we have $R_\mathfrak{q} = R_\mathfrak{p}$. Now note, that $\mathfrak{q}$ is a prime ideal of $\env{(R_\mathfrak{p})}$, and $(\env{(R_\mathfrak{p})})_\mathfrak{q} = (\env{R})_\mathfrak{q}$. So since $R_\mathfrak{p}$ is proxy small in $\dcat{\env{(R_\mathfrak{p})}}$, it is proxy small in $\dcat{(\env{R})_\mathfrak{q}}$. 

For the converse direction, assume $R$ is proxy small in $\dcat{\env{R}}$. That is there exists a complex $P$ in $\dcat{\env{R}}$, such that 
\[
\level_{\env{R}}^{\env{R}}(P) < \infty \quad\text{and}\quad \level_{\env{R}}^R(P) < \infty \quad\text{and}\quad \supp_{\env{R}}(P) = \supp_{\env{R}}(R)\,.
\]
Let $X \in \dfbcat{R}$. By Corollary~\ref{locallyCIProxySmall}, it is enough to show $X$ is proxy small in $\dcat{R}$. Any complex $Y$ in $\dcat{\env{R}}$ has a left and a right $R$-action. Thus $Y \Lotimes_R X$ has a left $R$-action through the left $R$-action of $Y$. This induces the exact functor
\[
- \Lotimes_R X \colon \dcat{\env{R}} \to \dcat{R}
\]
and, by \eqref{eq:levelExactFunctor}, one has 
\[
\level_R^{R \otimes_k X}(P \Lotimes_R X) < \infty \quad\text{and}\quad \level_R^X(P \Lotimes_R X) < \infty\,.
\]
The object $R \otimes_k X$ is a, possible infinite, direct sum of suspensions of $R$. Let $\Add(R)$ be the subcategory of all such complexes. Then
\[
R \otimes_k X \in \Add(R) \quad\text{and so}\quad \level_R^{\Add(R)}(P \Lotimes_R X) < \infty\,.
\]
In particular $P \Lotimes_R X$ has a finite resolution by projective modules. Since $P \Lotimes_R X$ is generated by $X$, it has finitely generated total homology. Thus $P \Lotimes_R X$ has a finite resolution by finitely generated projective modules, that is it is perfect. 

It remains to show $P \Lotimes_R X$ has the same support as $X$. A localizing subcategory generated by an object $X$ in $\dcat{R}$ is the smallest triangulated subcategory of $\dcat{R}$, that is closed under direct sums and contains $X$. By \cite[Theorem~2.8]{Neeman:1992b}, two complexes have the same support if and only if they generate the same localizing subcategory. Now since $P$ and $R$ have the same support over $\env{R}$, they have the same localizing subcategory in $\dcat{\env{R}}$. So $P \Lotimes_R X$ and $R \Lotimes_R X = X$ generate the same localizing subcategories in $\dcat{R}$ and so have the same support over $R$. 
\end{proof}

This characterization is similar to the characterization of a smooth ring: If $k$ is a field and $R$ a $k$-algebra essentially of finite type over $k$, then $R$ is smooth if and only if $R$ is small in $\dcat{\env{R}}$. 

For a generalization of this characterization, see \cite{Briggs/Iyengar/Letz/Pollitz:2020}. 

\begin{discussion}
As in Theorem~\ref{kAlgebraLCI} let $k$ be a field and $R$ a $k$-algebra of essentially finite type over $k$. If $R$ is locally complete intersection, and $Q \twoheadrightarrow R$ is a surjective map of $k$-algebras with $Q$ a regular ring and kernel $I$, then $R$ generates the small object $R \Lotimes_Q R$ in $\dcat{\env{R}}$. Adapting the argument of \cite[Theorem~9.1]{Dwyer/Greenlees/Iyengar:2006b} the generation time is bound below by
\[
\sup\set{\codim(R_\mathfrak{m})}{\mathfrak{m} \in \Max(R)} + 1 \leq \level_{\env{R}}^R(R \Lotimes_Q R)\,,
\]
and using \cite[Theorem~11.3]{Avramov/Buchweitz/Iyengar/Miller:2010} bound above by
\[
\level_{\env{R}}^R(R \Lotimes_Q R) \leq \sup\set{\height(I_\mathfrak{m})}{\mathfrak{m} \in \Max(R)} + 1\,.
\]
\end{discussion}

\subsection*{Acknowledgments} I thank my advisor Srikanth Iyengar for a lot of helpful discussions and reading many versions of this paper. I also thank Josh Pollitz for his interest in this work and pointing out the connection of Corollary~\ref{levelFaithfullyFlat} to the question posed in \cite[Remarks~9.6]{Dwyer/Greenlees/Iyengar:2006b}, and Jian Liu for his comments.

\end{document}